\documentclass[12pt]{amsart}
\usepackage{}
\usepackage{amscd,amsmath,amsthm,amssymb}
\usepackage{pstcol,pst-plot,pst-3d}%\usepackage[T1]{fontenc}
\usepackage{color}
\usepackage{pstricks}
\usepackage{tikz}
\usepackage{lineno,xcolor}
\DeclareSymbolFont{largesymbol}{OMX}{yhex}{m}{n}
\DeclareMathAccent{\Widehat}{\mathord}{largesymbol}{"62}

\newpsstyle{fatline}{linewidth=1.5pt}
\newpsstyle{fyp}{fillstyle=solid,fillcolor=verylight}
\definecolor{verylight}{gray}{0.97}
\definecolor{light}{gray}{0.9}
\definecolor{medium}{gray}{0.85}
\definecolor{dark}{gray}{0.6}

\def\G{{\mathcal G}}

\def\opn#1#2{\def#1{\operatorname{#2}}} % to make operators
\opn\chara{char} \opn\length{\ell} \opn\pd{pd} \opn\rk{rk}
\opn\projdim{proj\,dim} \opn\injdim{inj\,dim} \opn\rank{rank}
\opn\depth{depth} \opn\grade{grade} \opn\height{height}
\opn\embdim{emb\,dim} \opn\codim{codim}

\opn\Tr{Tr} \opn\bigrank{big\,rank}
\opn\superheight{superheight}\opn\lcm{lcm}
\opn\trdeg{tr\,deg}%\emph{
\opn\reg{reg} \opn\lreg{lreg} \opn\ini{in} \opn\lpd{lpd}
\opn\size{size} \opn\sdepth{sdepth}
\opn\link{link}\opn\fdepth{fdepth}\opn\lex{lex}
\opn\tr{tr}
\opn\type{type}
\opn\Borel{Borel}
\opn\cdeg{cdeg}

\opn\div{div} \opn\Div{Div} \opn\cl{cl} \opn\Cl{Cl}
\opn\Spec{Spec} \opn\Supp{Supp} \opn\supp{supp} \opn\Sing{Sing}
\opn\Ass{Ass} \opn\Min{Min}\opn\Mon{Mon}
\opn\Ann{Ann} \opn\Rad{Rad} \opn\Soc{Soc}
\opn\Im{Im} \opn\Ker{Ker} \opn\Coker{Coker} \opn\Am{Am}
\opn\Hom{Hom} \opn\Tor{Tor} \opn\Ext{Ext} \opn\End{End}
\opn\Aut{Aut} \opn\id{id}

\opn\nat{nat}
\opn\pff{pf}%   \pf exists already
\opn\Pf{Pf} \opn\GL{GL} \opn\SL{SL} \opn\mod{mod} \opn\ord{ord}
\opn\Gin{Gin} \opn\Hilb{Hilb}\opn\sort{sort}
\opn\PF{PF}\opn\Ap{Ap}
\opn\aff{aff} \opn
\con{conv} \opn\relint{relint} \opn\st{st}
\opn\lk{lk} \opn\cn{cn} \opn\core{core} \opn\vol{vol}  \opn\inp{inp} \opn\nilpot{nilpot}
\opn\link{link} \opn\star{star}\opn\lex{lex}\opn\set{set}
\opn\width{wd}
\opn\Fr{F}
\opn\QF{QF}
\opn\G{G}
\opn\type{type}\opn\res{res}
\opn\gr{gr}

\def\cdeg{deg}

\def\pot#1#2{#1[\kern-0.28ex[#2]\kern-0.28ex]}

\opn\dirlim{\underrightarrow{\lim}}
\opn\inivlim{\underleftarrow{\lim}}
\let\to=\rightarrow

\def\Implies{\ifmmode\Longrightarrow \else
	\unskip${}\Longrightarrow{}$\ignorespaces\fi}
\def\implies{\ifmmode\Rightarrow \else
	\unskip${}\Rightarrow{}$\ignorespaces\fi}
\def\iff{\ifmmode\Longleftrightarrow \else
	\unskip${}\Longleftrightarrow{}$\ignorespaces\fi}

\let\:=\colon
\newtheorem{Theorem}{Theorem}[section]
\newtheorem{Lemma}[Theorem]{Lemma}

\let\epsilon\varepsilon
\let\kappa=\varkappa
\def\qed{\ifhmode\textqed\fi
	\ifmmode\ifinner\quad\qedsymbol\else\dispqed\fi\fi}
\def\textqed{\unskip\nobreak\penalty50
	\hskip2em\hbox{}\nobreak\hfil\qedsymbol
	\parfillskip=0pt \finalhyphendemerits=0}
\def\dispqed{\rlap{\qquad\qedsymbol}}

\opn\dis{dis}
\def\pnt{{\raise0.5mm\hbox{\large\bf.}}}

\opn\Lex{Lex}

\begin{document}
	%\linenumbers
	\title { Freiman  $t$-spread  principal Borel ideals}
	
	\author {Guangjun Zhu$^{^*}$\!\!\!, Yakun Zhao and Yijun Cui }

	\address{Authors address:  School of Mathematical Sciences, Soochow
		University, Suzhou 215006, P.R. China}
	\email{zhuguangjun@suda.edu.cn(Corresponding author:Guangjun Zhu),
		\linebreak[4]1768868280@qq.com(Yakun Zhao), 237546805@qq.com(Yijun Cui). \linebreak[2]}

	\dedicatory{ }
	
	\begin{abstract}
	An equigenerated monomial ideal $I$ is a
Freiman ideal if $\mu(I^2)=\ell(I)\mu(I)-{\ell(I)\choose 2}$ where $\ell(I)$ is the analytic spread of $I$ and $\mu(I)$ is the least number of  monomial generators of $I$.  	
 Freiman ideals are special since there
exists an exact formula computing the least number of  monomial generators of any of their
powers. In this paper we  give a complete classification of Freiman $t$-spread principal Borel ideals.

	\end{abstract}
	
	\thanks{* Corresponding author}
	
	%		13H10   	Special types (Cohen-Macaulay, Gorenstein, Buchsbaum, etc.)
	%		13D02   	Syzygies, resolutions, complexes
	%		05E40   	Combinatorial aspects of commutative algebra
	%		16S36   	Ordinary and skew polynomial rings and semigroup rings
	
	%		14M25   	Toric varieties, Newton polyhedra [See also 52B20]
	%		13A02   	Graded rings
	%		13F20   	Polynomial rings and ideals; rings of integer-valued polynomials
	%		13A18   	Valuations and their generalizations
	%		06A11   	Algebraic aspects of posets
	
	
	\maketitle
	
	\setcounter{tocdepth}{1}
	%\tableofcontents

	\section*{Introduction}

	Let $K$ be a field and $S=K[x_1,\ldots,x_n]$ the polynomial ring in $n$ variables over $K$, and let $I\subset S$ be a monomial ideal.
We denote by $\mathcal{G}(I)$ the unique minimal set of monomial
generators of $I$ and by $\mu(I)$ the number of elements in  $\mathcal{G}(I)$.
It is  a very difficult problem to exactly compute  $\mu(I^k)$ for each integer $k\geq 2$.
If $I$ is generated by a regular sequence and $\mu(I)=m$, it is well known that $\mu(I^k)={k+m-1\choose m-1}$,
which is the maximal that $\mu(I^k)$ can reach.
At the other extreme, Eliahou et. al. \cite{EHS} constructed  a family of monomial ideals  such that $\mu(I)$ can be arbitrarily large but satisfy
 $\mu(I^2)=9$.

As a consequence of the well-known theorem due to Freiman (see \cite{Fre}) from additive
number theory,  Herzog et. al. showed in  \cite[Theorem 1.8]{HSZ} that if $I$ is an equigenerated monomial ideal, that is, all its generators are of the same degree,
then $\mu(I^2)\geq \ell(I)\mu(I)-{\ell(I)\choose 2}$,
 where $\ell(I)$ is the analytic spread of $I$. If the equality holds, then Herzog and Zhu \cite{HZ1} called this ideal $I$ to be a {\em Freiman} ideal (or simply Freiman).
What makes a Freiman ideal  interesting is the fact that $\mu(I^k)$, for any $k$, can be computed by an exact formula in terms of $\mu(I)$ and $\ell(I)$ (see \cite[Corollary 2.9]{HZ1}).
Several characterizations of Freiman ideals were provided in \cite{HHZ}. For example, it was shown  that $I$ is  Freiman
 if and only if $F(I)$  has minimal multiplicity if and only if  $F(I)$  has a $2$-linear resolution, where $F(I)$ is the fiber cone  of $I$.
 This  implies that the defining ideal of $F(I)$  is generated by quadrics. It is a very restrictive condition for ideals arising from combinatorial structures, which often guarantees strong combinatorial properties.

Freiman ideals in the classes of Hibi ideals, Veronese type ideals, matroid ideals,  sortable ideals, edge ideals of several
graphs,   cover ideals of  simple connected unmixed bipartite graphs, and  cover ideals of some classes of graphs such as trees, circulant graphs, and whiskered graphs have been studied (see \cite{DG,HHZ,HZ1,HZ2,ZZC}).

Squarefree monomial ideals play some important roles in Commutative Algebra, not
only for its intrinsic value but  for its strong connections to Combinatorics and Topology.
Recently, a generalization of the notion of squarefree ideal has been given
in \cite{EHQ} by the definition of $t$-spread monomial ideal.
Let  $t\geq 0$  be  an integer and  $u=x_{i_1}x_{i_2}\cdots x_{i_d}$  a monomial in  $S$ with $1\leq i_1\leq \cdots \leq i_d\leq n$. $u$ is called {\em $t$-spread} if $i_{j+1}-i_{j}\geq t$ for $1\leq j\leq d-1$.
By convention, a monomial of degree one is $t$-spread for any nonnegative integer $t$.
A monomial ideal $I\subset S$ is {\em $t$-spread} if it is generated by $t$-spread monomials.  Obviously, every monomial ideal is $0$-spread and every $t$-spread monomial ideal  with $t\geq 1$ is a  squarefree monomial ideal.
A monomial ideal $I$ is called {\em $t$-spread strongly stable} if it satisfies the following condition: for all $u\in \mathcal{G}(I)$ and $j\in supp(u)$, if $i<j$ and $x_i(u/x_j)$ is $t$-spread, then $x_i(u/x_j)\in I$. A monomial ideal $I\subset S$ is called {\em $t$-spread principal Borel} if there exists a monomial $u\in \mathcal{G}(I)$ such that $I=B_t(u)$, where $B_t(u)$ denotes the smallest $t$-spread strongly stable ideal which contains $u$.
It is obvious that any $0$-spread principal Borel ideal is a principal Borel ideal.
 Herzog and Zhu in \cite[Theorem 4]{HZ2} gave a complete characterization of Freiman principal Borel ideals  by considering their  sorted graphs.

The purpose of this paper is to  provide a complete characterization of Freiman $t$-spread principal Borel ideals for any integer $t\geq 1$.

Our paper is organized as follows. In Sect. 2, we recall some  basic
facts used in the following sections. In Sect. 3, we give a complete characterization of Freiman $t$-spread principal Borel ideals for any integer $t\geq 1$.
The results are as follows:

\begin{Theorem} \label{degree2}
Let $u=\prod\limits_{j=1}^{d}x_{i_j}$  be a t-spread monomial of degree $d$  in $S$.
\begin{itemize}
		\item[(1)]  If $d=1$, then  $B_{t}(u)$ is Freiman;
        \item[(2)]  If $d=2$, then  $B_t(u)$ is Freiman if and only if $i_{1}\leq 2$, or $i_1=3$ and $i_2=t+3$.
       \item[(3)]  If $d\geq 3$, $i_1=1$ and $t=1$, then $B_{t}(u)$ is Freiman if and only if one of the follows holds:
	\begin{itemize}
		\item[(i)]   $i_j=(j-1)t+1$ for any $j\in [d-2]$ and $i_{d-1}\in \{(d-2)t+1,(d-2)t+2\}$; 
		\item[(ii)] There exsits some $p\in [d-3]$  such that $i_j=(j-1)t+1$ for any $j\in [p]$,
  $i_j=(j-1)t+2$ for any $j\in [d-1]\setminus [p]$ and $i_{d}\in \{(d-1)t+2,(d-1)t+3\}$.
     \end{itemize}
    \item[(4)]  If $d\geq 3$, $i_1=1$ and $t\geq 2$, then $B_{t}(u)$ is Freiman if and only if one of the follows holds:
	\begin{itemize}
			\item[(i)]   $i_j=(j-1)t+1$ for any $j\in [d-2]$ and  $i_{d-1}\in \{(d-2)t+1,(d-2)t+2\}$, or $i_j=(j-1)t+1$ for any $j\in [d-2]$,
 $i_{d-1}=(d-2)t+3$ and $i_{d}=(d-1)t+3$;
		\item[(ii)] There exsits some $q\in [d-3]$  such that $i_j=(j-1)t+1$ for any $j\in [q]$,
  $i_j=(j-1)t+2$ for any $j\in [d-1]\setminus [q]$ and $i_{d}\in \{(d-1)t+2,(d-1)t+3\}$.
     \end{itemize}
    \item[(5)] If $d\geq 3$ and $i_1=2$, then $B_t(u)$ is Freiman if and only if $u=\prod\limits_{i=1}^{d}x_{(i-1)t+2}$ or $u=(\prod\limits_{i=1}^{d-1}x_{(i-1)t+2})x_{(d-1)t+3}$;
     \item[(6)] If $d\geq 3$ and $i_1\geq 3$, then  $B_{t}(u)$ is not Freiman.
     \end{itemize}
\end{Theorem}

\medskip
We greatfully acknowledge the use of the computer algebra system CoCoA (\cite{Co}) for our experiments.

\medskip

\section{Preliminaries}

We firstly recall some basic facts about the sorted graph of a monomial ideal from \cite{HZ2}. Sortable sets of monomials  have been introduced by Sturmfels \cite{St}.  Some basic properties  of sortable ideals are
discussed  in  \cite{EH}.

Let $K$ be a field and $S=K[x_1,\ldots,x_n]$ the polynomial ring in $n$ variables over $K$.
Let $d$ be a positive integer, $S_d$   the $K$-vector space generated by the monomials of degree $d$ in $S$, and take two monomials $u,v\in S_d$. We
write $uv=x_{i_1}x_{i_2}\cdots x_{i_{2d}}$ with $1\leq i_1\leq i_2\leq \cdots \leq i_{2d}\leq n$,  and define
$$u'=x_{i_1}x_{i_3}\cdots x_{i_{2d-1}},\quad \text{and} \quad v'=x_{i_2}x_{i_4}\cdots x_{i_{2d}}.$$
The pair $(u',v')$ is called the {\it sorting} of $(u,v)$. The map
\[
\sort: S_d\times S_d \to S_d\times S_d,\ (u, v)\mapsto (u', v')
\]
is called the {\em sorting operator}. For example, if $u=x_1^{2}x_3$
and $v=x_2x_3^2$, then $\sort(u,v)= (x_1x_2x_3, x_1x_3^2)$.
 A pair $(u,v)$ is called to be  {\em sorted} if $\sort(u,v)=(u,v)$ or
$\sort(u,v)=(v,u)$, otherwise it is called to be {\em unsorted}. Notice that $\sort(u,v)=\sort(v,u)$,  and that if $(u,v)$ is sorted, then
$u\geq_{lex} v$.

A subset $B$ of monomials in $S_d$ is called {\em sortable} if
$\sort(B\times B)\subset B\times B$.
A  monomial ideal $I\subset S$ is called a {\em sortable
ideal} if $\mathcal{G}(I)$ is a sortable set.

Given a sortable ideal $I$, we associate a finite simple graph $G_{I,s}$, the {\em sorted graph} $G_{I,s}$ of $I$,
  whose vertex set is  $\mathcal{G}(I)$ and  edge set is defined as follows:
\[
E(G_{I,s})=\{\{u,v\}\: \text{$u\neq v$, $(u,v)$  is sorted}\}.
 \]

It was shown in \cite[Proposition 3.1]{EHQ} that any  $t$-spread principal Borel ideal is a sortable ideal. Therefore, we may apply the following theorem  to check which of the $t$-spread principal Borel ideals
are Freiman.

\begin{Theorem}
\label{sortable}{\em(\cite[Theorem 3]{HZ2})}
Let $I\subset S$ be a sortable ideal. Then $I$ is  Freiman  if and only if the sorted graph $G_{I,s}$ is chordal. In particular, if $G_{I,s}$
contains an induced  $t$-cycle of length  $t\geq 4$, then $I$ is not Freiman.
\end{Theorem}

\medskip

\section{Freiman $t$-spread principal Borel ideals}

In this section, we will characterize all Freiman ideals among the  $t$-spread principal Borel ideals for any integer $t\geq 1$.

If $u=x_{i}$ for some $1\leq i\leq n$, then  $B_t(u)=(x_1,\ldots,x_i)$ by convention. It follows that  $B_t(u)$ is Freiman from \cite[Proposition 3.2 and Theorem 3.3]{HZ1}.
Therefore, from now on, we will always assume that  $d\geq 2$ and $t\geq 1$ are two integers.

\begin{Lemma}\label{simple}
	Let  $u=x_{i_1}x_{i_2}\cdots x_{i_d}$ be a $t$-spread monomial of degree $d$ such that $i_j=(j-1)t+1$ for $1\leq j\leq d-1$ and $i_d>(d-1)t+1$.  Then $B_t(u)$ is Freiman.
\end{Lemma}
\begin{proof}
Since $B_t(u)=\prod\limits_{j=1}^{d-1}x_{(j-1)t+1}(x_{(d-1)t+1},x_{(d-1)t+2},\ldots,x_{i_{d}-1},x_{i_d})$, it follows that the sorted graph $G_{B_t(u),s}$ is a complete graph. Hence $B_t(u)$ is Freiman by Theorem\ref{sortable}.
\end{proof}

\begin{Theorem} \label{degree2} Let $u=x_{i_{1}}x_{i_{2}}$ be a $t$-spread monomial of degree $2$. 
\begin{enumerate}
\item[(a)]   If $t=1$, then $B_t(u)$ is  Freiman if and only if $i_1\leq 2$;
\item[(b)] If $t\geq 2$, then $B_t(u)$ is  Freiman  if and only if  $i_1\leq 2$, or $i_1=3$ and $i_2=t+3$.
  \end{enumerate}	
\end{Theorem}
\begin{proof}
 We   distinguish into the following four cases:

	$(1)$ If $i_{1}=1$, then    $B_t(u)$ is Freiman by Lemma \ref{simple}.
	
	$(2)$ If $i_{1}=2$.  We will prove that  the sorted graph  $G_{B_t(u),s}$ does not contain an induced  cycle of length  $\geq 4$, thus  $B_t(u)$ is Freiman by Theorem \ref{sortable}.

Indeed, if there exists  an induced  $m$-cycle $C_m$ of length $m\geq 4$ in $G_{B_t(u),s}$. Let   $x_{k_{1}}x_{j_{1}},x_{k_{2}}x_{j_{2}},\ldots,x_{k_{m}}x_{j_{m}}$  be the vertices  of $C_m$ in clockwise order,
then $k_{1}\neq k_{3}$, $k_{2}\neq k_{4}$ and $k_i\in \{1,2\}$ for any $1\leq i\leq m$. If $m \geq 5$, then
there exist  at least three different elements $p,q,r$ in $\{1,2,\ldots,m\}$ such that  $k_{p}=k_{q}=k_{r}$.  It follows  that $C_m$ contains an induced  cycle $G$
with vertices $x_{k_{p}}x_{j_{p}},x_{k_{q}}x_{j_{q}}$ and $x_{k_{r}}x_{j_{r}}$. Note that $G$ is also an induced  subgraph of $G_{B_t(u),s}$,
contradicting our assumption that $C_m$  is an induced  cycle graph in $G_{B_t(u),s}$.

If $m=4$. We may assume $k_{1}=1$ without loss of generality, thus $k_{3}=2$, and $j_{1}>j_{3}$ because of no edge between the vertices $x_{k_{1}}x_{j_{1}}$ and $x_{k_{3}}x_{j_{3}}$.
 We consider  the following two cases:
	
$(i)$ If $k_{2}=1$, then  $k_{4}=2$  and $j_{3}\geq j_{2}>j_{4}$. It follows that $j_{1}>j_{4}$, thus there is no edge between vertices $x_{k_{1}}x_{j_{1}}$ and $x_{k_{4}}x_{j_{4}}$ of  the sorted graph  $G_{B_t(u),s}$, a contradiction.

$(ii)$ If $k_{2}=2$, then $k_{4}=1$ and $j_{3}\geq j_{4}>j_{2}$. Hence $j_{1}>j_{2}$. It follows that there exists no edge between the vertices $x_{k_{1}}x_{j_{1}}$ and $x_{k_{2}}x_{j_{2}}$ of  the sorted graph  $G_{B_t(u),s}$, a contradiction.

By the above arguments, we know that the sorted graph  $\mathcal{G}_{B_t(u),s}$ does not contain an induced  cycle of length  $\geq 4$.
	
	$(3)$ If $i_{1}=3$ and  $i_2=t+3$, then $\mathcal{G}(B_t(u))=\{x_{1}x_{t+1},x_{1}x_{t+2},x_{1}x_{t+3},x_{2}x_{t+2},x_{2}x_{t+3},\\
x_{3}x_{t+3}\}$. In particular, if $t=1$, then $\mathcal{G}(B_1(u))=\{x_{1}x_{2},x_{1}x_{3},x_{1}x_{4},x_{2}x_{3},x_{2}x_{4},
x_{3}x_{4}\}$. In this case, $G_{B_{t}(u),s}$ contains an induced  $4$-cycle
with vertices $x_{1}x_{2},x_{2}x_{3},x_{3}x_{4}$ and $x_{1}x_{4}$. This yields that $B_{t}(u)$ is not Freiman by Theorem \ref{sortable}.
 If $t\geq 2$, then
 the sorted graph $G_{B_{t}(u),s}$ is a subgraph of a complete graph with $6$ vertices obtained by deleting an edge $\{ x_{1}x_{t+3},x_{2}x_{t+2}\}$. Hence  $G_{B_{t}(u),s}$ is a chordal graph, it follows that  $B_{t}(u)$ is Freiman by Theorem \ref{sortable}.
	
$(4)$ If $i_{1}=3$ and $i_2\geq t+4$, or $i_1\geq 4$, then $G_{B_{t}(u),s}$ contains an induced  $4$-cycle
with vertices $ x_{1}x_{t+3},x_{3}x_{t+3},x_{2}x_{t+2}$ and $x_{2}x_{t+4}$. Theorem \ref{sortable}  implies that $B_{t}(u)$ is not Freiman.
\end{proof}

\begin{Theorem}\label{special1}
	Let $d\geq 3$  be  an  integer, $u=\prod\limits_{i=1}^{d}x_{(i-1)t+2}$  a $t$-spread monomial in $S$. Then $B_t(u)$ is Freiman.
\end{Theorem}
\begin{proof} Note that $\mathcal{G}(B_t(u))=A\cup \{u\}$, where $A=\{(\prod\limits_{j=1}^{i}x_{(j-1)t+1})(\prod\limits_{j=i+1}^{d}x_{(j-1)t+2})\mid 1\leq i\leq d\}$ and $\prod\limits_{j=\ell}^{m}x_j=1$ if $\ell> m$.
  By  simple calculations, one has the pair $(u, v)$ is sorted for any $v\in \mathcal{G}(B_t(u))$.
Let $v_1, v_2\in A$, we may assume $v_1=(\prod\limits_{j=1}^{a}x_{(j-1)t+1})(\prod\limits_{j=a+1}^{d}x_{(j-1)t+2})$,  $v_2=(\prod\limits_{j=1}^{b}x_{(j-1)t+1})(\prod\limits_{j=b+1}^{d}x_{(j-1)t+2})$, where  $1\leq a<b\leq d$.
 Thus $v_1v_2=(\prod\limits_{j=1}^{a}x_{(j-1)t+1}^2)(\prod\limits_{j=a+1}^{b}x_{(j-1)t+1}x_{(j-1)t+2})(\prod\limits_{j=b+1}^{d}x_{(j-1)t+2}^2)$.
  It follows that the pair $(v_1, v_2)$ is sorted. Hence the sorted graph
$G_{B_t(u), s}$ is a complete graph, the desired result follows from   Theorem \ref{sortable}.
\end{proof}

\begin{Theorem}\label{special2}
	Let $d\geq 3$  be  an  integer, $u=(\prod\limits_{i=1}^{d-1}x_{(i-1)t+2})x_{(d-1)t+3}$  a $t$-spread monomial of degree $d$ in $S$. Then $B_t(u)$ is Freiman.
\end{Theorem}
\begin{proof}  Note that $\mathcal{G}(B_t(u))=\{u_1\}\cup A\cup B\cup \{u\}$, where $u_1=\prod \limits_{j=1}^{d}x_{(j-1)t+1}$,   $A=\{(\prod \limits_{j=1}^{p}x_{(j-1)t+1})(\prod \limits_{j=p+1}^{d}x_{(j-1)t+2})\mid 0\leq p\leq d-1\}$, $B=\{(\prod \limits_{j=1}^{q}x_{(j-1)t+1})(\prod \limits_{j=q+1}^{d-1}x_{(j-1)t+2})\cdot x_{(d-1)t+3}\mid 1\leq q\leq d-1\}$ and $\prod\limits_{j=\ell}^{m}x_j=1$ if $\ell> m$.

It is enough to show that the sorted graph $G_{B_t(u),s}$ of $B_t(u)$ is a chordal graph  by Theorem \ref{sortable}.

Let  $C_\ell$ be  an induced  cycle of length $\ell\geq 4$ in  $G_{B_t(u),s}$ and we label its vertices $v_1,v_2,\ldots,v_\ell$ in clockwise order.
By simple calculations, we obtain   the pairs  $(u, v)$ and $(u_1, v)$ are  sorted for any $v\in \mathcal{G}(B_t(u))$. It follows that $u,u_1\not\in \{v_1,v_2,\ldots,v_\ell\}$.
Let $w_1,w_2$ be two vertices of $C_\ell$. We distinguish into the following three cases:

(i) If $w_1,w_2\in A$,  we may assume $w_1=(\prod \limits_{j=1}^{a}x_{(j-1)t+1})(\prod \limits_{j=a+1}^{d}x_{(j-1)t+2})$, $w_2=(\prod \limits_{j=1}^{b}x_{(j-1)t+1})(\prod \limits_{j=b+1}^{d}x_{(j-1)t+2})$ where $0\leq a<b\leq d-1$. Thus
$$w_1w_2=(\prod \limits_{j=1}^{a}x_{(j-1)t+1}^2)(\prod \limits_{j=a+1}^{b}x_{(j-1)t+1}x_{(j-1)t+2})(\prod \limits_{j=b+1}^{d}x_{(j-1)t+2}^2).$$
It follows that the pair $(w_1,w_2)$ is sorted.

(ii) If  $w_1,w_2\in B$, then, by an argument similar to (i), one has the pair $(w_1,w_2)$ is also  sorted.

(iii) If $w_1\in A$, $w_2\in B$, then $w_1=(\prod \limits_{j=1}^{a}x_{(j-1)t+1})(\prod \limits_{j=a+1}^{d}x_{(j-1)t+2})$, $w_2=(\prod \limits_{j=1}^{b}x_{(j-1)t+1})\\ \cdot(\prod \limits_{j=b+1}^{d-1}x_{(j-1)t+2})x_{(d-1)t+3}$
for some $a,b$.
Note that $x_{(d-1)t+2}$ is a  factor of $w_1$ and $x_{(d-1)t+3}$ is a  factor of $w_2$, it follows that the pair $(w_1,w_2)$ is also  sorted if and only if  $a\geq b$.

If $\ell\geq 5$, then there exist  at least three different elements $v_{i_1},v_{i_2},v_{i_3}\in A$ or $v_{i_1},v_{i_2},v_{i_3}\in B$.  It follows  that $C_\ell$ contains an induced  cycle $G$
with vertices $v_{i_1},v_{i_2}$ and $v_{i_3}$  by case (i) or (ii). Moreover $G$ is also an induced  subgraph of $G_{B_t(u),s}$,
contradicting our assumption that $C_\ell$  is an induced  cycle graph in $G_{B_t(u),s}$.

If $\ell=4$. Let $v_1\in A$, then $v_3\in B$ since there exists no edge  between the vertices $v_1$ and $v_3$.
 We prove the desired results when $v_2\in A$, and the case $v_2\in B$ can be shown by similar arguments.
 Let $v_2\in A$, then $v_4\in B$. We may assume that $v_1=(\prod \limits_{j=1}^{a}x_{(j-1)t+1})(\prod \limits_{j=a+1}^{d}x_{(j-1)t+2})$, $v_2=(\prod \limits_{j=1}^{b}x_{(j-1)t+1})(\prod \limits_{j=b+1}^{d}x_{(j-1)t+2})$,
 $v_3=(\prod \limits_{j=1}^{c}x_{(j-1)t+1})(\prod \limits_{j=c+1}^{d-1}x_{(j-1)t+2})x_{(d-1)t+3}$ and  $v_4=(\prod \limits_{j=1}^{e}x_{(j-1)t+1})(\prod \limits_{j=e+1}^{d-1}x_{(j-1)t+2})x_{(d-1)t+3}$, where $a\neq b$ and
  $c\neq e$. Note that
 $x_{(d-1)t+2}$ is a common factor of $v_1$ and  $v_2$, and  $x_{(d-1)t+3}$ is a  common  factor of $v_3$ and $v_4$. Hence $a\geq e$ and $b\geq c$ since $\{v_1,v_4\}$ and $\{v_2,v_3\}$  are  the two edges of  $C_\ell$.
 If $a>b$, then $a>c$, this implies that $\{v_1,v_3\}$  is an edge of  $C_\ell$.  If $b>a$, then $b>e$, this implies that $\{v_2,v_4\}$  is an edge of  $C_\ell$. Hence, in both cases, contradicting our assumption that
that $C_\ell$ is  an induced cycle of length $\ell$ in $G_{B_{t}(u),s}$.
We finished the proof.
\end{proof}

\begin{Theorem}\label{end}
	Let $d\geq 3$  be  an  integer, $u=\prod\limits_{j=1}^{d}x_{i_j}$ a $t$-spread monomial of degree $d$ in $S$.
\begin{itemize}
		\item[(1)]  If $i_1=2$, then $B_t(u)$ is Freiman if and only if $u=\prod\limits_{i=1}^{d}x_{(i-1)t+2}$ or $u=(\prod\limits_{i=1}^{d-1}x_{(i-1)t+2})x_{(d-1)t+3}$;
        \item[(2)] If $i_1\geq 3$, then  $B_{t}(u)$ is not Freiman.
\end{itemize}
\end{Theorem}
\begin{proof}
	(1)  First, if  $u=\prod\limits_{i=1}^{d}x_{(i-1)t+2}$ or $u=(\prod\limits_{i=1}^{d-1}x_{(i-1)t+2})x_{(d-1)t+3}$, then  $B_t(u)$ is Freiman by Theorems \ref{special1} and  \ref{special2}.
 In the remaining cases, we distinguish into the following two cases:

$(i)$ Let $d=3$ and  $u=x_2x_{i_2}x_{i_3}$.  If $i_2=t+2$ and $i_3\geq 2t+4$, then
$A\subset \mathcal{G}(B_t(u))$, where $A=\{x_1x_{t+2}x_{2t+2},x_1x_{t+2}x_{2t+4},x_1x_{t+1}x_{2t+3},x_2x_{t+2}x_{2t+3}\}$.
It's easy to see that the graph  with the vertex set $A$  is an induced $4$-cycle in $G_{B_{t}(u),s}$.
If $i_2\geq t+3$, then $B\subset \mathcal{G}(B_t(u))$, where $B=\{x_1x_{t+1}x_{2t+3},x_1x_{t+3}x_{2t+3},x_1x_{t+2}x_{2t+2},x_2x_{t+2}x_{2t+3}\}$.
Moreover,  the graph  with the vertex set $B$  is an induced $4$-cycle  in $G_{B_{t}(u),s}$.
In both cases, Theorem \ref{sortable} implies that $B_{t}(u)$ is not Freiman.

$(ii)$ Suppose $d\geq 4$ and $u=x_{2}\prod\limits_{j=2}^{d}x_{i_j}$. If there exists  some $i_j=(j-1)t+3$ for  $2\leq j\leq d-1$, then
$C\subset \mathcal{G}(B_t(u))$, where  $v=\prod\limits_{j=1}^{d-3}x_{(j-1)t+1}$ and $C=\{x_{(d-3)t+1}x_{(d-2)t+1}x_{(d-1)t+3}v,x_{(d-3)t+1}x_{(d-2)t+3}x_{(d-1)t+3}v,
x_{(d-3)t+1}x_{(d-2)t+2}x_{(d-1)t+2}v,\\ x_{(d-3)t+2}x_{(d-2)t+2}x_{(d-1)t+3}v\}$.
It's easy to see that the graph  with the vertex set $C$  is an induced $4$-cycle  in $G_{B_{t}(u),s}$.
If $i_j\geq (j-1)t+4$ for some  $2\leq j\leq d$, then $D\subset \mathcal{G}(B_t(u))$, where $D=\{x_{(d-3)t+1}x_{(d-2)t+2}x_{(d-1)t+2}v,x_{(d-3)t+1}x_{(d-2)t+2}x_{(d-1)t+4}v,\\ x_{(d-3)t+1}x_{(d-2)t+1}x_{(d-1)t+3}v,x_{(d-3)t+2}x_{(d-2)t+2}x_{(d-1)t+3}v\}$
and  $v=\prod\limits_{j=1}^{d-3}x_{(j-1)t+1}$.
Moreover,  the graph  with the vertex set $D$  is an induced $4$-cycle in $G_{B_{t}(u),s}$.
In both cases, Theorem \ref{sortable} implies that $B_{t}(u)$ is not Freiman.

(2)  Note that $A\subset \mathcal{G}(B_t(u))$, where $A=\{x_{1}x_{t+1}x_{2t+3}v,x_{2}x_{t+2}x_{2t+3}v,x_{1}x_{t+2}x_{2t+2}v,\\
x_{1}x_{t+3}x_{2t+3}v\}$, $v=\prod\limits_{j=4}^dx_{(j-1)t+3}$ and we stipulate $\prod\limits_{j=\ell}^dx_{(j-1)t+3}=1$ if $\ell>d$.
It's obvious  that the graph  with the vertex set $A$  is an induced $4$-cycle  in $G_{B_{t}(u),s}$.   Theorem \ref{sortable} implies that $B_{t}(u)$ is not Freiman.
\end{proof}

\begin{Lemma}\label{degree}
	Let $S=K[x_1,\ldots,x_n]$ and  $S'=S[x_{n+1},\ldots,x_{n+t}]$ be two polynomial rings over field $K$. Let  $u=x_1x_{i_2}\cdots x_{i_d}$ is a $t$-spread monomial of degree $d$ in $S'$ and the map
	\[
	\varphi_t:S'\longrightarrow S, \ u\longmapsto x_{i_2-t}x_{i_3-t}\cdots x_{i_d-t}.
	\]
	Then $B_t(u)$ is Freiman if and only if $B_t(\varphi_t(u))$ is Freiman.
\end{Lemma}
\begin{proof}
	Let	$u_1=x_1x_{i_2}\cdots x_{i_d}$, $u_2=x_1x_{j_2}\cdots x_{j_d}$ be two  $t$-spread monomials of degree $d$  in $S'$. From the
	definition of $\varphi_t$, one has  the pair $(u_1,u_2)$ is sorted in $S'$ if and only if the pair $(\varphi_t(u_1),\varphi_t(u_2))$ is sorted in $S$. It follows that the  sorted  graphs $G_{B_t(u),s}$ and $G_{B_t(\varphi_t(u)),s}$ are isomorphic. Hence $B_t(u)$ is Freiman if and only if $B_t(\varphi_t(u))$ is Freiman from Theorem~\ref{sortable}.
\end{proof}

	Let  $m$ be a positive integer, we set $[m]=\{1,\ldots,m\}$.

\begin{Theorem}\label{start}
	Let $d\geq 3$  be  an  integer, $u=x_1\prod\limits_{j=2}^{d}x_{i_j}$ a $t$-spread monomial of degree $d$ in $S$.
\begin{itemize}
 \item[(1)]  If  $t=1$, then $B_{t}(u)$ is Freiman if and only if one of the follows holds:
	\begin{itemize}
		\item[(i)]   $i_j=(j-1)t+1$ for any $j\in [d-2]$ and $i_{d-1}\in \{(d-2)t+1,(d-2)t+2\}$;
		\item[(ii)] There exsits some $p\in [d-3]$  such that $i_j=(j-1)t+1$ for any $j\in [p]$,
  $i_j=(j-1)t+2$ for any $j\in [d-1]\setminus [p]$ and $i_{d}\in \{(d-1)t+2,(d-1)t+3\}$.
\end{itemize}
    \item[(2)]  If $t\geq 2$, then $B_{t}(u)$ is Freiman if and only if one of the follows holds:
	\item[(i)]   $i_j=(j-1)t+1$ for any $j\in [d-2]$ and  $i_{d-1}\in \{(d-2)t+1,(d-2)t+2\}$, or $i_j=(j-1)t+1$ for any $j\in [d-2]$,
 $i_{d-1}=(d-2)t+3$ and $i_{d}=(d-1)t+3$;
		\item[(ii)] There exsits some $q\in [d-3]$  such that $i_j=(j-1)t+1$ for any $j\in [q]$,
  $i_j=(j-1)t+2$ for any $j\in [d-1]\setminus [q]$ and $i_{d}\in \{(d-1)t+2,(d-1)t+3\}$.
\end{itemize}
\end{Theorem}
\begin{proof}The desired results hold  by repeatedly appling Lemmas \ref{simple}, \ref{degree}  and  Theorems  \ref{degree2} $\sim$  \ref{end}.
\end{proof}

	\medskip
	\hspace{-6mm} {\bf Acknowledgments}
	
	\vspace{3mm}
	\hspace{-6mm}  This research is supported by the National Natural Science Foundation of China (No.11271275) and  by foundation of the Priority Academic Program Development of Jiangsu Higher Education Institutions.

	\medskip

\end{document}